\documentclass[final,1p,times,authoryear]{elsarticle}
\usepackage[english]{babel}

\usepackage[bitstream-charter]{mathdesign}
\usepackage[T1]{fontenc}

\usepackage{amsthm}

\usepackage{color}

\usepackage{caption}
\usepackage{tabularx}

\usepackage[]{hyperref}

\newtheorem{theorem}{Theorem}

\newtheorem{proposition}[theorem]{Proposition}

\biboptions{round}

\journal{JMPS}

\usepackage{latexsym}
\usepackage{amsmath}
\usepackage{graphics}
\usepackage{color}
\usepackage{tikz}
\usepackage{pgfplots}
\usepackage{subcaption}

\begin{document}

\begin{frontmatter}



\title{Sharp-interface limits for brittle fracture  via the inverse-deformation formulation}


\author[label1]{Timothy J.\ Healey}
\ead{tjh10@cornell.edu}
\author[label2]{Roberto Paroni}
\ead{roberto.paroni@unipi.it}
\author[label3,label4]{Phoebus Rosakis}
\ead{rosakis@uoc.gr}
 
\address[label1]{Department of Mathematics, Cornell University, Ithaca, NY 14853, USA}
\address[label2]{Department of Civil and Industrial Engineering, University of Pisa, Pisa 56122, Italy}
\address[label3]{Department of Mathematics and Applied Mathematics, University of Crete, Heraklion 70013 Crete, Greece}
\address[label4]{Institute for Applied and Computational Mathematics, Foundation of Research and Technology-Hellas, Heraklion 70013, Crete, Greece}

\begin{abstract}
We derive sharp-interface models for one-dimensional brittle fracture via the inverse-deformation approach.  
Methods of $\Gamma$-convergence are employed to obtain the singular limits of previously proposed models.  The 
latter feature a local, non-convex stored energy of inverse strain, augmented by small interfacial energy, formulated in terms of the 
inverse-strain gradient.  They predict spontaneous fracture with exact crack-opening discontinuities, without the use of damage (phase) fields or pre-existing cracks; 
crack faces are endowed with a thin layer of surface energy.  The models obtained herewith inherit the same properties, except that surface energy is now concentrated at the crack faces. Accordingly, we construct energy-minimizing configurations. For a composite bar with a breakable layer, our results predict a pattern of equally spaced cracks whose number is given as an increasing function of applied load.
\end{abstract}




\end{frontmatter}


\tableofcontents
\clearpage


\section{Introduction}

The modelling of  nucleation and growth of fracture in solids is notoriously difficult, due to the fact that cracks are typically represented by discontinuities of the deformation mapping, corresponding to infinite strains.  On the other hand, the inverse of the deformation mapping can be extended as a continuous, piecewise-smooth mapping in such situations. This simple observation is the basis for our recent work on brittle fracture in 
one-dimensional filaments, cf. \cite{RHA}.  The model features a local, non-convex stored energy function of inverse  stretch, augmented by small higher-gradient 
energy, with the latter formulated in terms of the inverse-strain gradient.  It predicts spontaneous fracture with  discontinuous deformations, as a global bifurcation from the homogeneously deformed state.  Moreover, neither damage (phase) fields nor pre-existing cracks are employed. 

The inverse formulation of classical nonlinear elasticity is due to \cite{S}, and the general strain-gradient version was obtained 
by \cite{CS}. The idea of fracture as a two-well phase transition is due to \cite{T}, where the second well is located at a strain 
going to infinity.  However, as implied above, this causes great difficulties---both analytical and numerical; the inclusion 
of higher-gradient energy does not regularize the problem, but instead suppresses fracture altogether, cf.  \cite{RHA}.  In contrast, starting  with a Lennard-Jones type stored energy in direct variables, Fig.~\ref{fig1a}, the inverse-deformation formulation naturally replaces the potential well at 
infinity with one at zero inverse-stretch; the problem now indeed has the same appearance as a two-well phase-transition model with inverse stored energy energy as in { Fig.~\ref{fig1b}.  
Moreover, the {higher-gradient  energy in the inverse formulation gives rise to surface energy in a small neighborhood of crack faces, cf. \cite{RHA}. Last but not least, despite the presence of higher inverse-strain gradients, cracks in equilibrium deformations are exact discontinuities, in contrast with ones arising in phase-field models; see \cite{RHA} for details.

The present work is a follow-up to that of \cite{RHA} and \cite{GH}.  The latter employs the same model presented in the former while also accounting for interaction with a pseudo-rigid elastic 
foundation, the latter of which is constrained to undergo only homogeneous elastic deformations, cf. \cite{VHR}. The analysis in \cite{GH} addresses the nontrivial problem of crack development, not arising in \cite{RHA}.  Our goal here is to find sharp-interface models in the singular limit as the small parameter characterizing the inverse higher-gradient energy  goes to zero.  Of course, our main tool is $\Gamma$-convergence.  The resulting models feature cracks as discontinuous jumps of the deformation, endowed with surface energy on crack faces,  in precise analogy with the Griffith criterion, cf. \cite{G}.

 In Section 2, we summarize the aforementioned models characterized by small interfacial energy.  In both cases, we show that 
a homogeneous solution in compression always corresponds to the unique energy minimizer. Our main results for tensile loading
are presented in Section 3.  After the usual rescaling, the sharp-interface model associated with \cite{RHA} follows as a textbook problem (for phase transitions), and the two energy-minimizing configurations, each characterized by a single end crack, are readily obtained.  

The model from \cite{GH} presents a more challenging problem.  First, due to the interaction with the pseudo-rigid elastic foundation, the energy functional necessarily involves the inverse deformation as well as its first and second derivatives. Moreover, the usual rescaling leads to a model greatly overemphasizing the interaction constant between the brittle rod and the foundation.  A more realistic model emerges by first scaling the interaction constant before rescaling the energy.  The novel sharp-interface model we obtain features a nontrivial interplay between surface energy and the elastic foundation.  As such, energy minimizing configurations are not obvious.
 
We explore energy-minimizing configurations in Section 4.  We first note that any such inverse-deformation field must have 
at least one finite jump in derivative.  We then prove that minimizers, which come in pairs, are characterized remarkably by a uniform, repeating pattern. The repetition number coincides with the number of cracks and is obtained explicitly via an elementary first-derivative test.  We close with a concrete example. The minimizing inverse configurations are determined and plotted along with their associated deformation maps. We also show that the number of cracks in a minimizing configuration increases with increasing load. This and the equally spaced arrangement of cracks capture experimental observations by \cite{M2} in composite biological fibers with an external brittle layer.

\section{Formulation}

We consider a Lennard-Jones type energy density, e.g., 
$W(F)=(1-F^{-1})^{2},$ cf.\ Figure \ref{fig1a}, where $f:[0,1]\to [0,\lambda ]$ is 
the deformation and $F:=f'>0$ denotes the stretch. Here $\lambda>0$ denotes the prescribed average stretch. Our forthcoming analysis does not depend on this specific choice of $W$. Indeed, any continuously differentiable, non-negative function $W:(0,\infty )\to [0,\infty )$ having 
the same qualitative properties, viz.,  convex on $(0,1)$, a single, isolated potential well at 
$F=1$ with $W(\ref{eq1})=0,$ a single inflection point at some point $F>1,$ and a 
horizontal asymptote as $F\to +\infty,$ will suffice. In any case, the inverse stored 
energy is defined by
\begin{equation}
\label{eq1}
W^{\ast }(H):=HW(H^{-1}),
\end{equation}
where $h:[0,\lambda ]\to [0,1]$ is the inverse deformation and $H:={h}'=1/F$ 
denotes the inverse stretch. For the example mentioned above, (\ref{eq1}) yields 
$W^{\ast }(H)=H(1-H)^{2},$ cf.\ Figure \ref{fig1b}. Noting that $W^{\ast }(0)\to 0$ 
as $H\searrow 0,$ we may accordingly extend the domain of $W^{\ast }$ to 
$[0,\infty ).$ Thus, $W^{\ast }(H)>0$ for $H\ne 0,\mbox{1,}$ with $W^{\ast 
}(0)=W^{\ast }(1)=0.$ As already noted in \cite{RHA}, $W^{\ast }$ has the 
appearance of a two-well potential as in models for phase transitions. Here, 
$H=0$ represents the cracked or broken phase. 

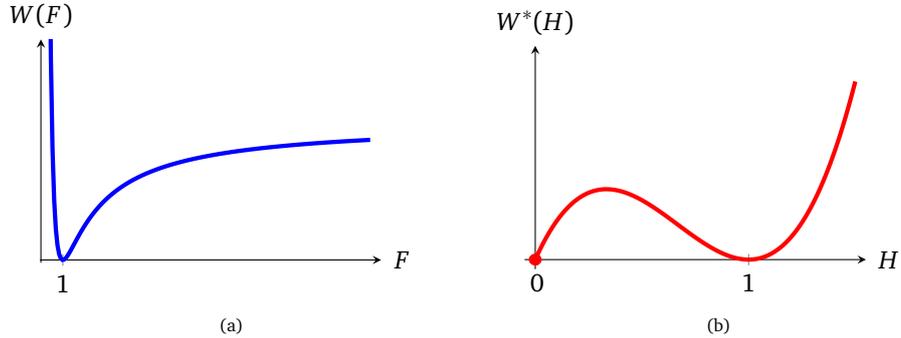
\begin{figure}[h]
  \centering
  \begin{subfigure}[b]{0.45\textwidth}
\begin{tikzpicture}[scale=1]
\begin{axis}[
    width=1\textwidth,
    height=.23\textheight,
    axis x line=center,
    axis y line=center,
    xmin = -0.05, xmax = 15.5,
    ymin = -0.02, ymax = 1.6,
    xtick={1}, xticklabels={1},
    ytick={0}, yticklabels={0},
    domain = 0:15,
    samples = 200,
    xlabel = {$F$},
    ylabel = {$W(F)$},
    every axis x label/.style={at={(ticklabel* cs:1.01)},anchor=west,},
    every axis y label/.style={at={(ticklabel* cs:1.01)},anchor=south,},      ]
    \addplot[ultra thick, blue] {(1-1/x)^2};
\end{axis}
\end{tikzpicture}
\caption{}
\label{fig1a}
\end{subfigure}%
 \quad
 \begin{subfigure}[b]{0.45\textwidth}
\begin{tikzpicture}[scale=1]
\begin{axis}[
    width=1\textwidth,
    height=.23\textheight,
    axis x line=center,
    axis y line=center,
    xmin = -0.05, xmax = 1.55,
    ymin = -0.02, ymax = 0.45,
    xtick={0.01,1}, xticklabels={0,1},
    ytick={0}, yticklabels={0},
    domain = 0:1.5,
    samples = 200,
    xlabel = {$H$},
    ylabel = {$W^*(H)$},
    every axis x label/.style={at={(ticklabel* cs:1.01)},anchor=west,},
    every axis y label/.style={at={(ticklabel* cs:1.01)},anchor=south,},      ]
    \addplot[ultra thick, red] {x*(1-x)^2};
    \begin{scope}[yscale=4]
    \draw[red, ultra thick, fill] (5,5) circle (2);
    \end{scope}
\end{axis}
\end{tikzpicture}
\caption{}
\label{fig1b}
\end{subfigure}
\caption{(a) Lennard Jones-type stored energy density $W$. (b) Corresponding 
inverse stored energy function $W^*$. The states corresponding to $H<0$ are inaccessible, rendering $W^*$ a two-well potential with minima at $H=0,1$.}
\label{fig1}
\end{figure}%

Following  \cite{S}, the elastic energy can be expressed as
\begin{equation}
\label{eq2}
\int_0^1 {W({f}'(x))dx=\int_0^\lambda {W^{\ast }} } ({h}'(y))dy.
\end{equation}
We use $x\in[0,1]$ as the reference coordinate and $y\in[0,\lambda]$ as the deformed coordinate throughout, so that $x\mapsto f(x)\in[0,\lambda]$ and $y\mapsto h(y)\in[0,1] $.
Due to the properties of $W^{\ast },$ the minimization of \eqref{eq2} yields an uncountable infinity of solutions, some of which are
charcterized by an infnite number of cracks. 
The situation is comparable to the problem treated in  \cite{E}. As a remedy, we introduce an interfacial energy according to
\begin{equation}
\label{eq3}
E_{\varepsilon } [H]=\int_0^\lambda {\Big( {\frac{\varepsilon 
^{2}}{2}({H}')^{2}+W^{\ast }(H)} \Big)} dy,
\end{equation}
where $\varepsilon >0$ is a small parameter. As shown in \cite{RHA}, this 
first-gradient energy not only regularizes the problem, but also introduces 
small surface energy at crack faces. It is worth observing that \eqref{eq3} is not 
the same as the standard strain-gradient model
\begin{equation}
\label{std}
\int_0^1 {\Big( {\frac{\varepsilon^{2}}{2}({F}')^{2}+W(F)} \Big)} dx.
\end{equation}
The minimization of the latter entails working in a subset of $H^{1}(0,1).$ 
By embedding, all elements of the latter are continuous, viz., $F$ is 
continuous, implying no fracture. While the same holds for \eqref{eq3} (in fact 
$H$ is shown to be $C^{1}$ in \cite{RHA}, the inverse stretch field 
readily accommodates $H=0$ corresponding to $F\sim \infty$, indicating 
fracture. Finally, we append the compatibility condition
\begin{equation}
\label{eq4}
\int_0^\lambda {Hdy=1} 
\end{equation}
to the formulation (\ref{eq3}). 

Following \cite{GH}, we also consider the interaction with a pseudo-rigid elastic 
foundation via an addition to \eqref{eq3}, which is necessarily expressed in terms of the 
inverse deformation $h$:
\begin{equation}
\label{eq5}
U_{\varepsilon } [h]=\int_0^\lambda {\Big( {\frac{\varepsilon 
^{2}}{2}({h}'')^{2}+W^{\ast }({h}')+\frac{k{h}'}{2}(y-\lambda h)^{2}} 
\Big)} dy, \quad {h}' \ge  0,
\end{equation}
where $k>0$ represents the interaction stiffness between the pseudo-rigid elastic 
foundation and the brittle specimen. In this case, we drop (\ref{eq4}) and 
supplement (\ref{eq5}) with the boundary conditions
\begin{equation}
\label{eq6}
h(0)=0,\quad h(\lambda )=1.
\end{equation}
We finish this section with a result for compressive loadings:

\begin{proposition}\label{prop1}
If $0<\lambda \le  1,$ then 
$$
\min \big\{E_{\varepsilon}[H]: H\in H^1(0,\lambda), H\ge 0 \mbox{ a.e., } \int_0^\lambda H \,dy=1\big\}
$$ 
and 
$$
\min \big\{U_{\varepsilon}[h]: h\in H^2(0,\lambda), h'\ge 0 \mbox{ a.e., } h(0)=h(\lambda)=1 \big\}
$$ 
are attained by the homogeneous inverse deformation  $h(y)=y/\lambda$, with constant inverse stretch $H=1/\lambda$. In particular, the corresponding minima are $E_\varepsilon[1/\lambda]=U_\varepsilon[y/\lambda]=\lambda W^\ast (1/\lambda)=W(\lambda).$ 
\end{proposition}
\begin{proof}
Let $W_{c}^{\ast }$ denote the convex envelope of $W^{\ast },$ cf.\ Figure 
\ref{fig2}. Then by \eqref{eq5}, \eqref{eq6} and Jensen's inequality, we find
\begin{eqnarray*}
 U_{\varepsilon} [h]&\ge& \int_0^\lambda {\Big( {\frac{\varepsilon 
^{2}}{2}({h}'')^{2}+W_{c}^{\ast } ({h}')+\frac{k{h}'}{2}(y-\lambda h)^{2}} 
\Big)} dy \\ 
&\ge& \int_0^\lambda {W_{c}^{\ast } 
({h}')} dy\geqslant \lambda W_{c}^{\ast } \Big( {\frac{1}{\lambda 
}\int_0^\lambda {{h}'} dy} \Big)=\lambda W_{c}^{\ast }(1/\lambda ), \\ 
 \end{eqnarray*}
for all $h\in H^{2}(0,\lambda )$. Since $W$ is non-negative and convex in $(0,1)$, it follows that $W_{c}^{\ast }(H)= W^{\ast }(H)$ for all $H\ge 1$, cf. Figure \ref{fig2}. From this remark and observing that  $1/\lambda \ge 1,$ we  deduce that 
$$
U_{\varepsilon} [h]\ge\lambda W^\ast(1/\lambda )  
=E_\varepsilon[1/\lambda]=U_\varepsilon[y/\lambda],$$ for all $h\in H^{2}(0,\lambda )$.  
Clearly the same argument 
applies to the minimizer of $E_{\varepsilon}$.
\end{proof}

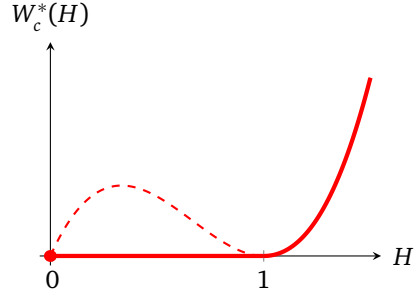
\begin{figure}[h]
\centering
 \begin{tikzpicture}[scale=1]
\begin{axis}[
    width=.45\textwidth,
    height=.23\textheight,
    axis x line=center,
    axis y line=center,
    xmin = -0.05, xmax = 1.55,
    ymin = -0.02, ymax = 0.45,
    xtick={0.01,1}, xticklabels={0,1},
    ytick={0}, yticklabels={0},
    domain = 0:1.5,
    samples = 200,
    xlabel = {$H$},
    ylabel = {$W^*_c(H)$},
    every axis x label/.style={at={(ticklabel* cs:1.01)},anchor=west,},
    every axis y label/.style={at={(ticklabel* cs:1.01)},anchor=south,},      ]
    \addplot[domain = 1:1.5,ultra thick, red] {x*(1-x)^2};
    \addplot[thick, red, dashed] {x*(1-x)^2};
    \addplot [ultra thick, red ] coordinates {(0,0) (1,0) };
    \begin{scope}[yscale=4]
    \draw[red, ultra thick, fill] (5,5) circle (2);
    \end{scope}
\end{axis}
\end{tikzpicture}
\caption{Convex envelope of $W^{\ast }.$ }
\label{fig2}
\end{figure}%

\section{Sharp-Interface Models via  $\Gamma$-Convergence}

We first consider (\ref{eq3}), (\ref{eq4}), and rescale the former according to
\begin{equation}
\label{eq7}
I_{\varepsilon } [H]:=\frac{1}{\varepsilon }E_{\varepsilon } 
[H]=\int_0^\lambda {\Big( {\frac{\varepsilon 
}{2}({H}')^{2}+\frac{1}{\varepsilon }W^{\ast }(H)} \Big)} dy, \quad
\int_0^\lambda {Hdy=1},\quad \mbox{for\, }H\geqslant 0,
\end{equation}
where $\varepsilon >0,$ with $I_{\varepsilon } :=+\infty $ otherwise. We also 
assume the physically reasonable growth condition 
\begin{equation*}
W(F)\ge C \frac 1 F \mbox{\, for\, all\, }0<F\le \frac 1M,
\end{equation*}
for constants $C>0$ and  $M>1.$ 
From (1), this is equivalent to
\begin{equation}
\label{eq8}
W^*(H)\geqslant CH^{2}\mbox{\, for\, all\, }H\geqslant M,
\end{equation}
for constants $C>0$ and  $M>1.$ For instance, (\ref{eq8}) is satisfied for 
$W^{\ast }(H)=H(1-H)^{2},$ mentioned previously. Moreover, (\ref{eq8}) is benign in 
our setting since it entails only highly compressive behavior. In any case, 
(\ref{eq7}) is a textbook problem for $\Gamma \mbox{-}$convergence as $\varepsilon 
\searrow 0,$ e.g., \cite{ABM}, \cite{B}; the general result for bounded domains in 
$\mathbb{R}^{n}$ was first presented in \cite{M}. In what follows, $PC(0,\lambda 
)$ denotes the space of piecewise constant functions on $(0,\lambda ),$ with 
$\# D(u)$ denoting the number of discontinuities of $u\in PC(0,\lambda ).$ We 
quote the following result:
\begin{proposition}\label{prop2}
For $\lambda \geqslant 1,$ the family 
$\{I_{\varepsilon } \}_{\varepsilon >0}$  $\Gamma$-converges  
in the strong $L^{1}(0,\lambda )$ topology to
\begin{equation}\label{eq9}
I[H]=
\begin{cases}
C_{W^{\ast }} \,  \# D(H) & H\in PC(0,\lambda ), \quad H\in \{0,1\}  \mbox{ a.e. 
in\, }(0,\lambda),\\ 
&\mbox{and } \displaystyle \int_0^\lambda {Hdy=1,}\\
+\infty &\mbox{otherwise,}
\end{cases}
\end{equation}
where 
$C_{W^{\ast }} :=\int_0^1 {\sqrt {2W^{\ast }(\tau )} } d\tau.$ 
Moreover, if $\{\varepsilon_{n} \}$ and $\{H_{n} \}\subset H^{1}(0,\lambda 
)$ are sequences such that $\varepsilon_{n} \searrow 0$ and 
$\{I_{\varepsilon_{n} } [H_{n} ]\}$ is uniformly bounded, then $\{H_{n} 
\}$ has a convergent subsequence in $L^{1}(0,\lambda ),$ and every limit point 
$H_{\ast }$ is an element of $PC(0,\lambda )$ with $H_{\ast } \in \{0,1\} 
\mbox{ a.e.\  in }(0,1).$
\end{proposition}

The minimum of (\ref{eq9}) is easily deduced. We first note that $H\in PC(0,\lambda )$ has finite  energy, i.e., $I[H]<+\infty$, if and only if  $H\in \{0,1\}$  a.e.\  in $(0,\lambda)$ and the set $\{y\in (0,\lambda): H(y)=1\}$ has measure equal to 1.
Thus, if $\lambda =1,$ $I$ is uniquely 
minimized by $H\equiv 1$. This corresponds to the 
undeformed specimen whose energy is $I=0$. For $\lambda >1,$ the condition $I[H]<+\infty$ necessarily implies that the set $\{y\in (0,\lambda): H(y)=0\}$ has measure $\lambda-1$, which in turn implies that $\# D(H)\ge 1$.
It follows that $I$ is minimized by a single crack and the minimum is $I=C_{W^{\ast }}$. 

Figure \ref{fig3a} depicts the inverse-stretch field  $H=1$ in $(0,1)$ with $H=0$  in $(1,\lambda ),$
which minimizes $I$. The associated deformation mapping is shown in Figure \ref{fig3b}; a single crack occurs at the right extreme.

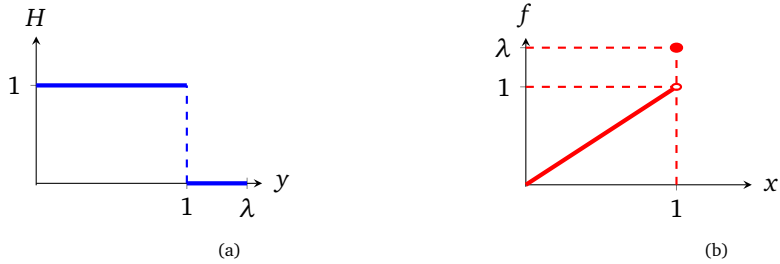
\begin{figure}[h]
\centering
  \begin{subfigure}[b]{0.45\textwidth}
\begin{tikzpicture}[scale=1]
\begin{axis}[
    width=.75\textwidth,
    height=.18\textheight,
    axis x line=center,
    axis y line=center,
    xmin = 0, xmax = 1.5,
    ymin = -0.02, ymax = 1.5,
    xtick={0,1,1.4}, xticklabels={0,1,$\lambda$},
    ytick={0,1}, yticklabels={0,1},
    domain = 0:1.5,
    samples = 200,
    xlabel = {$y$},
    ylabel = {$H$},
    every axis x label/.style={at={(ticklabel* cs:1.01)},anchor=west,},
    every axis y label/.style={at={(ticklabel* cs:1.01)},anchor=south,},      ]
    \addplot[domain = 0:1,ultra thick, blue] {1};
    \addplot[domain =1:1.4,ultra thick, blue] {0};
    \addplot [thick, blue, dashed ] coordinates {(1,0) (1,1) };
\end{axis}
\end{tikzpicture}
\caption{}
\label{fig3a}
\end{subfigure}%
 \quad
 \begin{subfigure}[b]{0.45\textwidth}
\begin{tikzpicture}[scale=1]
\begin{axis}[
    width=.75\textwidth,
    height=.18\textheight,
    axis x line=center,
    axis y line=center,
    xmin = 0, xmax = 1.5,
    ymin = -0.02, ymax = 1.5,
    xtick={0,1}, xticklabels={0,1},
    ytick={0,1,1.4}, yticklabels={0,1,$\lambda$},
    domain = 0:1.5,
    samples = 200,
    xlabel = {$x$},
    ylabel = {$f$},
    every axis x label/.style={at={(ticklabel* cs:1.01)},anchor=west,},
    every axis y label/.style={at={(ticklabel* cs:1.01)},anchor=south,},      ]
    \addplot[domain = 0:0.98, ultra thick, red] {x};
    \draw[red, thick] (axis cs: 1, 1) circle [radius=3];
    \draw[red, ultra thick, fill] (axis cs: 1, 1.4) circle [radius=3];
    \addplot [thick, red, dashed ] coordinates {(1,0) (1,0.98) };
    \addplot [thick, red, dashed ] coordinates {(1,1.03) (1,1.4) };
    \addplot [thick, red, dashed ] coordinates {(0,1.4) (1,1.4) };
    \addplot [thick, red, dashed ] coordinates {(0,1) (0.97,1) };
\end{axis}
\end{tikzpicture}
\caption{}
\label{fig3b}
\end{subfigure}
\caption{(a) An inverse strain that minimizes $I$. (b) The associated broken configuration.}
\label{fig3}
\end{figure}%

Several admissible inverse-stretch fields are depicted in Figure \ref{fig4}: (A) is the same as Figure \ref{fig3} corresponding 
to a crack at the extreme right; (B) shows a single crack at the extreme left. Like configuration (A), this field minimizes $I$. 
The inverse-strain field shown in (C) features cracks at both extremes; inverse-stretch field (D) has an open crack on the 
interior $(0,\lambda ),$ while (E) and (F) depict 
inverse-stretch fields having both interior and end cracks. In each of the cases (C)-(F), 
it follows that $I\ge 2C_{W^{\ast }}.$ Finally, we point out that the single-end-failure prediction 
agrees with the results obtained in \cite{RHA} for $\varepsilon >0.$

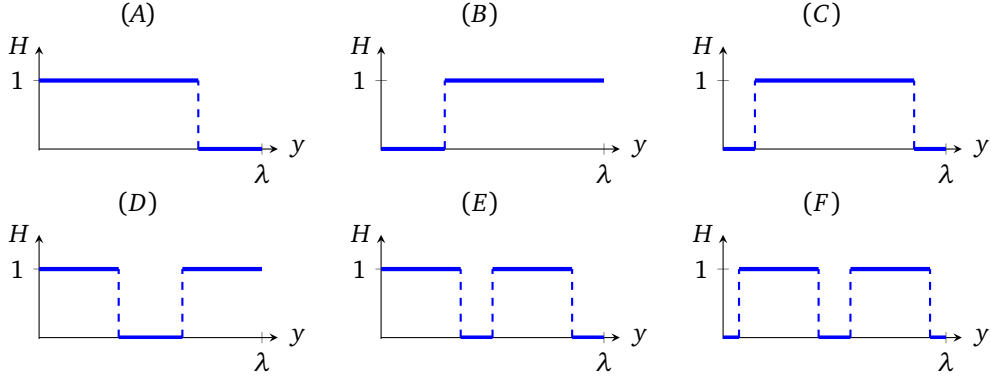
\begin{figure}[h]
\centering
\begin{tikzpicture}[scale=1]
\begin{axis}[
    width=.35\textwidth,
    height=.15\textheight,
    axis x line=center,
    axis y line=center,
    xmin = 0, xmax = 1.5,
    ymin = -0.02, ymax = 1.5,
    xtick={1.4}, xticklabels={$\lambda$},
    ytick={0,1}, yticklabels={0,1},
    domain = 0:1.5,
    samples = 200,
    xlabel = {$y$},
    ylabel = {$H$},
    every axis x label/.style={at={(ticklabel* cs:1.01)},anchor=west,},
    every axis y label/.style={at={(ticklabel* cs:1.01)},anchor=east,},      ]
    \addplot[domain = 0:1,ultra thick, blue] {1};
    \addplot[domain =1:1.4,ultra thick, blue] {0};
    \addplot [thick, blue, dashed ] coordinates {(1,0) (1,1) };
\end{axis}
\draw (1.3,1.8) node  {$(A)$};
\begin{scope}[xshift=4.5cm]
\begin{axis}[
    width=.35\textwidth,
    height=.15\textheight,
    axis x line=center,
    axis y line=center,
    xmin = 0, xmax = 1.5,
    ymin = -0.02, ymax = 1.5,
    xtick={1.4}, xticklabels={$\lambda$},
    ytick={0,1}, yticklabels={0,1},
    domain = 0:1.5,
    samples = 200,
    xlabel = {$y$},
    ylabel = {$H$},
    every axis x label/.style={at={(ticklabel* cs:1.01)},anchor=west,},
    every axis y label/.style={at={(ticklabel* cs:1.01)},anchor=east,},      ]
    \addplot[domain = 0.4:1.4,ultra thick, blue] {1};
    \addplot[domain =0:0.4,ultra thick, blue] {0};
    \addplot [thick, blue, dashed ] coordinates {(0.4,0) (0.4,1) };
\end{axis}
\draw (1.3,1.8) node  {$(B)$};
\end{scope}
\begin{scope}[xshift=9cm]
\begin{axis}[
    width=.35\textwidth,
    height=.15\textheight,
    axis x line=center,
    axis y line=center,
    xmin = 0, xmax = 1.5,
    ymin = -0.02, ymax = 1.5,
    xtick={1.4}, xticklabels={$\lambda$},
    ytick={0,1}, yticklabels={0,1},
    domain = 0:1.5,
    samples = 200,
    xlabel = {$y$},
    ylabel = {$H$},
    every axis x label/.style={at={(ticklabel* cs:1.01)},anchor=west,},
    every axis y label/.style={at={(ticklabel* cs:1.01)},anchor=east,},      ]
    \addplot[domain = 0.2:1.2,ultra thick, blue] {1};
    \addplot[domain =1.2:1.4,ultra thick, blue] {0};
    \addplot[domain =0:0.2,ultra thick, blue] {0};
    \addplot [thick, blue, dashed ] coordinates {(0.2,0) (0.2,1) };
    \addplot [thick, blue, dashed ] coordinates {(1.2,0) (1.2,1) };
\end{axis}
\draw (1.3,1.8) node  {$(C)$};
\end{scope}
\begin{scope}[xshift=0cm, yshift=-2.5cm]
\begin{axis}[
    width=.35\textwidth,
    height=.15\textheight,
    axis x line=center,
    axis y line=center,
    xmin = 0, xmax = 1.5,
    ymin = -0.02, ymax = 1.5,
    xtick={1.4}, xticklabels={$\lambda$},
    ytick={0,1}, yticklabels={0,1},
    domain = 0:1.5,
    samples = 200,
    xlabel = {$y$},
    ylabel = {$H$},
    every axis x label/.style={at={(ticklabel* cs:1.01)},anchor=west,},
    every axis y label/.style={at={(ticklabel* cs:1.01)},anchor=east,},      ]
    \addplot[domain = 0:0.5,ultra thick, blue] {1};
    \addplot[domain = 0.9:1.4,ultra thick, blue] {1};
    \addplot[domain =0.5:0.9,ultra thick, blue] {0};
    \addplot [thick, blue, dashed ] coordinates {(0.5,0) (0.5,1) };
    \addplot [thick, blue, dashed ] coordinates {(0.9,0) (0.9,1) };
\end{axis}
\draw (1.3,1.8) node  {$(D)$};
\end{scope}
\begin{scope}[xshift=4.5cm, yshift=-2.5cm]
\begin{axis}[
    width=.35\textwidth,
    height=.15\textheight,
    axis x line=center,
    axis y line=center,
    xmin = 0, xmax = 1.5,
    ymin = -0.02, ymax = 1.5,
    xtick={1.4}, xticklabels={$\lambda$},
    ytick={0,1}, yticklabels={0,1},
    domain = 0:1.5,
    samples = 200,
    xlabel = {$y$},
    ylabel = {$H$},
    every axis x label/.style={at={(ticklabel* cs:1.01)},anchor=west,},
    every axis y label/.style={at={(ticklabel* cs:1.01)},anchor=east,},      ]
    \addplot[domain = 0:0.5,ultra thick, blue] {1};
    \addplot[domain = 0.7:1.2,ultra thick, blue] {1};
    \addplot[domain =0.5:0.7,ultra thick, blue] {0};
    \addplot[domain =1.2:1.4,ultra thick, blue] {0};
    \addplot [thick, blue, dashed ] coordinates {(0.5,0) (0.5,1) };
    \addplot [thick, blue, dashed ] coordinates {(0.7,0) (0.7,1) };
    \addplot [thick, blue, dashed ] coordinates {(1.2,0) (1.2,1) };
\end{axis}
\draw (1.3,1.8) node  {$(E)$};
\end{scope}
\begin{scope}[xshift=9cm, yshift=-2.5cm]
\begin{axis}[
    width=.35\textwidth,
    height=.15\textheight,
    axis x line=center,
    axis y line=center,
    xmin = 0, xmax = 1.5,
    ymin = -0.02, ymax = 1.5,
    xtick={1.4}, xticklabels={$\lambda$},
    ytick={0,1}, yticklabels={0,1},
    domain = 0:1.5,
    samples = 200,
    xlabel = {$y$},
    ylabel = {$H$},
    every axis x label/.style={at={(ticklabel* cs:1.01)},anchor=west,},
    every axis y label/.style={at={(ticklabel* cs:1.01)},anchor=east,},      ]
    \addplot[domain = 0.1:0.6,ultra thick, blue] {1};
    \addplot[domain = 0.8:1.3,ultra thick, blue] {1};
    \addplot[domain =0:0.1,ultra thick, blue] {0};
    \addplot[domain =0.6:0.8,ultra thick, blue] {0};
    \addplot[domain =1.3:1.4,ultra thick, blue] {0};
    \addplot [thick, blue, dashed ] coordinates {(0.1,0) (0.1,1) };
    \addplot [thick, blue, dashed ] coordinates {(0.6,0) (0.6,1) };
    \addplot [thick, blue, dashed ] coordinates {(0.8,0) (0.8,1) };
    \addplot [thick, blue, dashed ] coordinates {(1.3,0) (1.3,1) };
\end{axis}
\draw (1.3,1.8) node  {$(F)$};
\end{scope}
\end{tikzpicture}
\caption{Several inverse strains having finite energy $I$.}
\label{fig4}
\end{figure}%

Directly employing the same energy scaling in (\ref{eq5}), we arrive at a model 
overemphasizing the strength of the stiffness $k.$ Accordingly, we first 
assume $k=\varepsilon \mu .$ The rescaling of the energy now leads to
\begin{equation}
\label{eq10}
V_{\varepsilon } [h]:=\frac{1}{\varepsilon }U_{\varepsilon } 
[h]=\int_0^\lambda {\Big( {\frac{\varepsilon 
}{2}({h}'')^{2}+\frac{1}{\varepsilon }W^{\ast }({h}')+\frac{\mu 
{h}'}{2}(y-\lambda h)^{2}} \Big)} dy,\quad {h}'>0,
\end{equation}
subject to (\ref{eq6}), where $V_{\varepsilon } :=+\infty $ otherwise. 

In what follows, let $PC^{1}(0,\lambda)$ denote the space of continuous 
functions having a piecewise constant derivative on $(0,\lambda ),$ with $\# 
D({u}')$ denoting the number of discontinuities of ${u}'.$ 

\begin{proposition}\label{prop3}
For $\lambda \geqslant 1,$ the family 
$\{V_{\varepsilon } \}_{\varepsilon >0}$ $\Gamma$-converges in the 
strong $H^{1}(0,\lambda )$ topology to 
\begin{equation}\label{eq11}
V[h]=\begin{cases}
 \displaystyle C_{W^*} \# D({h}')+\frac{\mu }{2}\int_0^\lambda {{h}'(y-\lambda h)^{2}dy} & h\in PC^{1}(0,\lambda ),\\ 
 & {h}'\in \{0,1\} \mbox{ a.e.\ in }(0,\lambda ),\\
 & \mbox{and } h(0)=0,\quad h(\lambda )=1,\\
 +\infty &\mbox{otherwise,} 
\end{cases}
\end{equation}
where $C_{W^{\ast }} :=\int_0^1 {\sqrt {2W^{\ast }(\tau )} } d\tau .$ 
Moreover, if $\{\varepsilon_{n} \}$ and $\{h_{n} \}\subset H^{2}(0,\lambda 
)$ are sequences such that $\varepsilon_{n} \searrow 0$ and $\{V_{\varepsilon_{n} } [h_{n} ]\}$ is uniformly bounded, then $\{h_{n} \}$ has a convergent 
subsequence in $H^{1}(0,\lambda ),$ and every limit point $h_{\ast } $ is an 
element of $PC^{1}(0,\lambda )$ with ${h}'_{\ast } \in \{0,1\} 
\mbox{ a.e.\ in\, }(0,\lambda)$, and satisfies the boundary conditions $h_{\ast}(0)=0$ and $h_{\ast }(\lambda )=1$.
\end{proposition}

\begin{proof}
Let $V_{\varepsilon }^{o}$ denote the functional \eqref{eq10} with $\mu =0.$ It can 
be shown that the $\Gamma$-convergence part of Proposition \ref{prop2} is also 
valid in the strong $L^{2}(0,\lambda )$ topology, cf. \cite{B}. By virtue of \eqref{eq6}, 
we then deduce that the family $\{V_{\varepsilon }^{o} \}_{\varepsilon >0}$ 
$\Gamma$-converges in the strong $H^{1}(0,\lambda )$ topology 
according to \eqref{eq11} with $\mu =0.$ Because the $\varepsilon$-independent 
functional $\int_0^\lambda {\frac{\mu {h}'}{2}(y-\lambda h)^{2}dy}$ is continuous 
on $H^{1}(0,\lambda ),$ we conclude that \eqref{eq11} holds $(\mu >0).$

For the compactness result, first note that if $\{V_{\varepsilon_{n} } [h_{n} ]\}$ is 
uniformly bounded, then so is $\{V_{\varepsilon_{n} }^{o} [h_{n} ]\}$ for 
$\{h_{n} \}\subset H^{2}(0,\lambda ).$ From the second part of Proposition \ref{prop2}  
while again employing the boundary conditions \eqref{eq6}, we deduce that $\{h_{n} 
\}\subset H^{2}(0,\lambda )$ has a convergent subsequence in 
$W^{1,1}(0,\lambda ).$ Finally, as shown in \cite{CFL}, the growth condition \eqref{eq8} 
then implies that $\{h_{n} \}$ has a convergent subsequence in 
$H^{1}(0,\lambda )$ as well. 
\end{proof}

\section{Energy-minimizing configurations}
We now explore minimizing configurations for the limiting functional \eqref{eq11}. 
If $\lambda =1,$ then $V$ is minimized by $h\equiv y$, which 
represents the unbroken specimen with corresponding energy $V=0$. If $\lambda >1,$ observe that the integral 
term in \eqref{eq11} does not vanish: ${h}'=0$ a.e.\ in $(0,\lambda )$ is 
incompatible with the boundary conditions, and $h=y/\lambda$ on any subset 
of non-zero measure violates ${h}'\in \{0,1\}$ a.e.\ in  $(0,\lambda ).$ 
Hence, there is a competition between the surface and foundation energies, and it follows that $h'$ must suffer at least one disconituity. 

To construct minimum-energy configurations for $\lambda >1,$ we start by 
defining two functions on a partial segment of length $\ell \leqslant 
\lambda$: 
\begin{equation}\label{eq12}
h_{(1)} =\begin{cases}
y &  0\leqslant y<\ell /\lambda, \\
 \ell /\lambda  & \ell /\lambda \leqslant y\leqslant \ell,
 \end{cases}
\quad
h_{(2)} =\begin{cases}
 0 & 0\leqslant y<\ell -\ell/\lambda, \\
 y-\ell (\lambda -1)/\lambda & \ell -\ell/\lambda \leqslant 
y\leqslant \ell.
\end{cases}
\end{equation}
Clearly, $h_{(\alpha )} \in PC^{1}[0,\ell ]$ with ${h}'_{(\alpha )} \in 
\{0,1\}, \alpha =1,2,$ as shown in Figure \ref{fig5}.

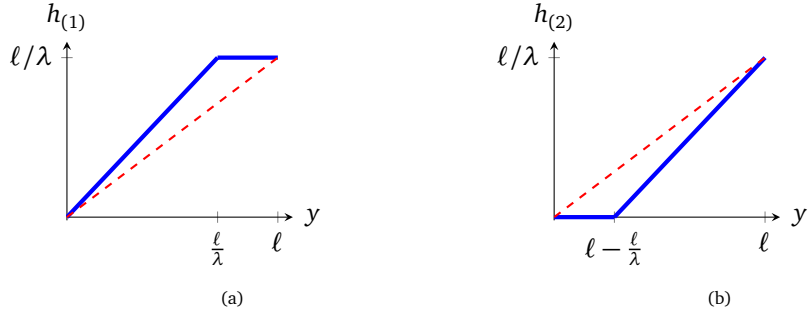
\begin{figure}[h]
\centering
  \begin{subfigure}[b]{0.45\textwidth}
  \begin{tikzpicture}[scale=1]
\begin{axis}[
    width=.75\textwidth,
    height=.20\textheight,
    axis x line=center,
    axis y line=center,
    xmin = 0, xmax = 1.5,
    ymin = -0.02, ymax = 1.1,
    xtick={0,1,1.4}, xticklabels={0,$\frac{\ell}\lambda$, $\ell$},
    ytick={0,1}, yticklabels={0,$\ell/\lambda$},
    domain = 0:1.5,
    samples = 200,
    xlabel = {$y$},
    ylabel = {$h_{(1)}$},
    every axis x label/.style={at={(ticklabel* cs:1.01)},anchor=west,},
    every axis y label/.style={at={(ticklabel* cs:1.01)},anchor=south,},      ]
    \addplot[domain = 0:1,ultra thick, blue] {x};
    \addplot[domain =1:1.4,ultra thick, blue] {1};
    \addplot [thick, red, dashed ] coordinates {(0,0) (1.4,1) };
\end{axis}
\end{tikzpicture}
\caption{}
\label{fig5a}
\end{subfigure}%
 \quad
 \begin{subfigure}[b]{0.45\textwidth}
\begin{tikzpicture}[scale=1]
\begin{axis}[
    width=.75\textwidth,
    height=.20\textheight,
    axis x line=center,
    axis y line=center,
    xmin = 0, xmax = 1.5,
    ymin = -0.02, ymax = 1.1,
    xtick={0,0.4,1.4}, xticklabels={0,$\ell -\frac{\ell}\lambda$, $\ell$},
    ytick={0,1}, yticklabels={0,$\ell/\lambda$},
    domain = 0:1.5,
    samples = 200,
    xlabel = {$y$},
    ylabel = {$h_{(2)}$},
    every axis x label/.style={at={(ticklabel* cs:1.01)},anchor=west,},
    every axis y label/.style={at={(ticklabel* cs:1.01)},anchor=south,},      ]
    \addplot[domain = 0.4:1.4, ultra thick, blue] {x-0.4};
    \addplot [ultra thick, blue] coordinates {(0,0) (0.4,0) };
    \addplot [ thick, red, dashed ] coordinates {(0,0) (1.4,1) };
 \end{axis}
\end{tikzpicture}
\caption{}
\label{fig5b}
\end{subfigure}
\caption{(a) The graph  of the inverse strain $h_{(1)}$. (b) The graph  of $h_{(2)}$.}
\label{fig5}
\end{figure}%

Moreover, each of these yield 
the same energy according to (\ref{eq11}) (defined on $[0,\ell ]),$ viz.,
\begin{equation}
\label{eq13}
V=C_{W^*} +\frac{\mu (\lambda -1)^{2}}{6\lambda^{3}}\ell^{3}.
\end{equation}
The idea is to use (\ref{eq12})$_{1,2}$ and (\ref{eq13}) to build minimum-energy 
configurations. For instance, if $\ell =\lambda,$ then obviously \eqref{eq13} yields 
$V=C_{W^*} +\mu (\lambda -1)^{2}/6.$ Next, consider two segments of length $\ell 
_{1}$ and $\ell_{2}$ with $\ell_{1} +\ell_{2} =\lambda.$ We use 
\eqref{eq12}$_{1}$ and \eqref{eq12}$_{2}$ consecutively (or vice-versa), inducing only two 
discontinuities in ${h}';$ the corresponding configurations are depicted in 
Figure \ref{fig6}.

\begin{figure}[h]
\centering
  \begin{subfigure}[b]{0.45\textwidth}
  \begin{tikzpicture}[scale=1]
\begin{axis}[
    width=1\textwidth,
    height=.20\textheight,
    axis x line=center,
    axis y line=center,
    xmin = 0, xmax = 3.5,
    ymin = -0.02, ymax = 2.55,
    xtick={0,1,1.4,1.97,3.4}, xticklabels={0,$\frac{\ell_1}\lambda$, $\ell_1$,$\ \ \lambda-\frac{\ell_2}\lambda$, $\lambda$},
    ytick={0,1,2.43}, yticklabels={0,$\ell_1/\lambda$,1},
    domain = 0:1.5,
    samples = 200,
    xlabel = {$y$},
    ylabel = {$h$},
    every axis x label/.style={at={(ticklabel* cs:1.01)},anchor=west,},
    every axis y label/.style={at={(ticklabel* cs:1.01)},anchor=south,},      ]
    \addplot[domain = 0:1,ultra thick, blue] {x};
    \addplot[domain =1:1.97,ultra thick, blue] {1};
    \addplot [thick, red, dashed ] coordinates {(0,0) (3.4,2.43) };
    \addplot [ultra thick, blue ] coordinates {(1.97,1) (3.4,2.43) };
\end{axis}
\end{tikzpicture}
\caption{}
\label{fig6a}
\end{subfigure}%
 \quad
 \begin{subfigure}[b]{0.45\textwidth}
\begin{tikzpicture}[scale=1]
\begin{axis}[
    width=1\textwidth,
    height=.20\textheight,
    axis x line=center,
    axis y line=center,
    xmin = 0, xmax = 3.5,
    ymin = -0.02, ymax = 2.55,
    xtick={0,0.57,2,3,3.4}, xticklabels={0,$\ell_2-\frac{\ell_2}\lambda$, $\ell_2$,$\ell_2-\frac{\ell_1}\lambda\ \ $, $\ \lambda$},
    ytick={0,1.43,2.43}, yticklabels={0,$\ell_2/\lambda$,1},
    domain = 0:1.5,
    samples = 200,
    xlabel = {$y$},
    ylabel = {$h$},
    every axis x label/.style={at={(ticklabel* cs:1.01)},anchor=west,},
    every axis y label/.style={at={(ticklabel* cs:1.01)},anchor=south,},      ]
    \addplot[domain = 0.57:2, ultra thick, blue] {x-0.57};
    \addplot [ultra thick, blue] coordinates {(0,0) (0.57,0) };
    \addplot [ultra thick, blue] coordinates {(2,1.43) (3,2.43)};
    \addplot [ultra thick, blue] coordinates {(3,2.43) (3.4,2.43)};
    \addplot [thick, red, dashed ] coordinates {(0,0) (3.4,2.43)};
 \end{axis}
\end{tikzpicture}
\caption{}
\label{fig6b}
\end{subfigure}
\caption{(a) The graph of $h_{(1)}$ followed by $h_{(2)}$. (b) The graph of $h_{(2)}$ followed by $h_{(1)}$.}
\label{fig6}
\end{figure}
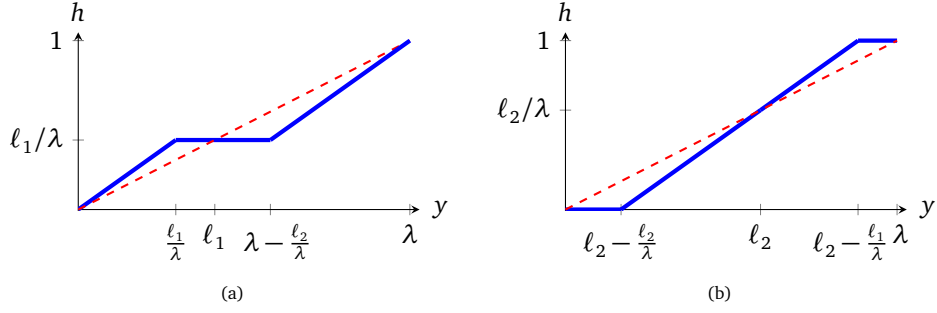%

In either case, from (\ref{eq11})-(\ref{eq13}), we find
\begin{equation}
\label{eq14}
V=2C_{W^*} +\frac{\mu (\lambda -1)^{2}}{6\lambda^{3}}[\ell_{1}^{3} +(\lambda 
-\ell_{1} )^{3}].
\end{equation}
The first-derivative test with respect to $\ell_{1} $ reveals $\ell 
_{1}^{2} -(\lambda -\ell_{1} )^2=0$, that is $\ell_{1} =\lambda /2=\ell 
_{2} .$ It is easy to see that this delivers the minimum value of (\ref{eq14}), 
given by
\begin{equation}
\label{eq15}
V=2C_{W^*} +\frac{\mu (\lambda -1)^{2}}{6\,2^2}.
\end{equation}
Something similar holds in the general case. First, assume a partition of 
$n$ segments of length $\ell_{j} , j=1,2,...,n,$ with 
$\sum\limits_{j=1}^n {\ell_{j} } =\lambda .$ We follow the same pattern as 
in Figure \ref{fig6}, viz., with alternating segments of $h_{(1)}$ and $h_{(2)}$ as 
specified in (\ref{eq12}). Then
\begin{equation}
\label{eq16}
V=nC_{W^*} +\frac{\mu (\lambda -1)^{2}}{6\lambda^{3}}[\sum\limits_{j=1}^{n-1} 
{\ell_{j}^{3} } +(\lambda -\sum\limits_{j=1}^{n-1} {\ell_{j} } )^{3}].
\end{equation}
The first-derivative test now implies $\ell_{j}^{2} =(\lambda 
-\sum\limits_{j=1}^{n-1} {\ell_{j} } )^{2}$ for  $j=1,2,...,n-1$,
which leads to
$
\ell_{1} =\ell_{2} =...=\ell_{n} =\frac\lambda n.
$
The corresponding minimum energy is given by
\begin{equation}
\label{eq17}
V=V_{n} :=nC_{W^*} +\frac{\mu (\lambda -1)^{2}}{6 n^{2}}.
\end{equation}

For any real number $x\ge 1$, we define 
an integer $\left[\kern-0.15em\left[ x 
\right]\kern-0.15em\right]$ as follows: let $m\leqslant x\leqslant m+1,$ 
i.e., $m$ denotes the integer part of $x$, then $\left[\kern-0.15em\left[ x 
\right]\kern-0.15em\right]=m$ if $V_{m} \leqslant V_{m+1}$ and 
$\left[\kern-0.15em\left[ x \right]\kern-0.15em\right]=m+1$ if $V_{m+1} 
\leqslant V_{m}.$ If $0 < x < 1$, we set $\left[\kern-0.15em\left[ x 
\right]\kern-0.15em\right]= 1$.

\begin{theorem}
For fixed $C_{W^*} ,\mu >0$ and $\lambda >1,$ the energy 
minimizing configuration for \eqref{eq11} is given by $n$ consecutive, alternating 
versions of \eqref{eq12}$_{1,2}$, with each segment of length $\ell =\lambda 
/n,$ where
\begin{equation}\label{nlambda}
n=\left[\kern-0.25em\left[ {\left( {\frac{\mu (\lambda -1)^{2}}{3C_{W^*} }} 
\right)^{1/3}} \right]\kern-0.25em\right].
\end{equation}
Moreover, there are two such equivalent configurations - one ``starting'' 
with \eqref{eq12}$_{1}$ and the other with \eqref{eq12}$_{2}$ as in Figure \ref{fig6}. 
\end{theorem}

As an example, consider the inverse stored energy $W^{\ast }(H)=H(1-H)^{2}$
with $\mu=200$. We then find
$$
C_{W^{\ast }}=\sqrt{2}\int_0^1\sqrt{\tau(1-\tau)^2}\, d\tau=2\sqrt{2}\int_0^1 (1-s^2)s^2\, ds=\frac{4}{15}\sqrt{2},
$$
and by means of \eqref{nlambda}, for any fixed $\lambda$, we can determine the number of cracks $n$.
For instance, for $\lambda=1.5$ we have that
$$
x={\left( {\frac{\mu (\lambda -1)^{2}}{3C_{W^*} }} 
\right)^{1/3}}=3.5355.
$$
Setting $m=3$, i.e., the largest integer smaller than $x$, we deduce from \eqref{eq17} that
$$
V_3=2.0753 \quad \mbox{and}\quad  V_4=2.0293.
$$
Hence, $n=\left[\kern-0.15em\left[ x \right]\kern-0.15em\right]=4$.  The minimizing configurations are depicted in Figures \ref{fig7a},\ref{fig7b}. In each case, we infer an associated deformation map $f$ by plotting the inverse of $h$ on each of the connected sets contained in 
$\{y\in[0,\lambda]:h'(h)>0\}$, as shown in Figures \ref{fig7c},\ref{fig7d}, respectively.
\begin{figure}[h]
\centering
  \begin{subfigure}[b]{0.45\textwidth}
  \begin{tikzpicture}[scale=1]
\begin{axis}[
    width=1\textwidth,
    height=.20\textheight,
    axis x line=center,
    axis y line=center,
    xmin = 0, xmax = 3.8,
    ymin = -0.02, ymax = 3.2,
    xtick={0,3.6}, xticklabels={0,$\lambda$},
    ytick={0,3}, yticklabels={0,1},
    domain = 0:1.5,
    samples = 200,
    xlabel = {$y$},
    ylabel = {$h$},
    every axis x label/.style={at={(ticklabel* cs:1.01)},anchor=west,},
    every axis y label/.style={at={(ticklabel* cs:1.01)},anchor=south,},      ]
    \addplot [thick, red, dashed ] coordinates {(0,0) (3.6,3) };
    \addplot [ultra thick, blue ] coordinates {(0,0) (0.75,0.75) (1.05,0.75) (2.55,2.25) (2.85,2.25) (3.6,3)  };
\end{axis}
\end{tikzpicture}
\caption{}
\label{fig7a}
\end{subfigure}%
 \quad
 \begin{subfigure}[b]{0.45\textwidth}
  \begin{tikzpicture}[scale=1]
\begin{axis}[
    width=1\textwidth,
    height=.20\textheight,
    axis x line=center,
    axis y line=center,
    xmin = 0, xmax = 3.8,
    ymin = -0.02, ymax = 3.2,
    xtick={0,3.6}, xticklabels={0,$\lambda$},
    ytick={0,3}, yticklabels={0,1},
    domain = 0:1.5,
    samples = 200,
    xlabel = {$y$},
    ylabel = {$h$},
    every axis x label/.style={at={(ticklabel* cs:1.01)},anchor=west,},
    every axis y label/.style={at={(ticklabel* cs:1.01)},anchor=south,},      ]
    \addplot [thick, red, dashed ] coordinates {(0,0) (3.6,3) };
    \addplot [ultra thick, blue ] coordinates {(0,0) (0.15,0) (1.65,1.5) (1.95,1.5) (3.45,3) (3.6,3)  };
\end{axis}
\end{tikzpicture}
\caption{}
\label{fig7b}
\end{subfigure}
\newline  
  \begin{subfigure}[b]{0.45\textwidth}
  \begin{tikzpicture}[scale=1]
\begin{axis}[
    width=1\textwidth,
    height=0.85\textwidth,
    axis x line=center,
    axis y line=center,
    xmin = -0.1, xmax = 3.8,
    ymin = -0.1, ymax = 3.8,
    xtick={0,3}, xticklabels={0,1},
    ytick={0,3.6}, yticklabels={0,$\lambda$},
    domain = 0:1.5,
    samples = 200,
    xlabel = {$x$},
    ylabel = {$f$},
    every axis x label/.style={at={(ticklabel* cs:1.01)},anchor=west,},
    every axis y label/.style={at={(ticklabel* cs:1.01)},anchor=south,},      ]
        \draw[red, thick] (axis cs: 0.75, 0.75) circle [radius=4];
   \draw[red, ultra thick, fill] (axis cs: 0, 0) circle [radius=3];
   \addplot [ultra thick, red ] coordinates {(0,0) (0.73,0.73)};
    \draw[red, ultra thick, fill] (axis cs: 3,3.6) circle [radius=4];
    \addplot [ultra thick, red ] coordinates {(2.27,2.87) (3,3.6)};
   \draw[red, thick] (axis cs: 2.25, 2.85) circle [radius=4];
   \draw[red, thick] (axis cs: 0.75, 1.05) circle [radius=4];
   \draw[red, thick] (axis cs: 2.25, 2.55) circle [radius=4];
   \addplot [ultra thick, red ] coordinates { (0.77,1.07) (2.23,2.53)};
\end{axis}
\end{tikzpicture}
\caption{}
\label{fig7c}
\end{subfigure}%
%
\quad
  \begin{subfigure}[b]{0.45\textwidth}
  \begin{tikzpicture}[scale=1]
\begin{axis}[
    width=1\textwidth,
    height=0.85\textwidth,
    axis x line=center,
    axis y line=center,
    xmin = -0.1, xmax = 3.8,
    ymin = -0.1, ymax = 3.8,
    xtick={0,3}, xticklabels={0,1},
    ytick={0,3.6}, yticklabels={0,$\lambda$},
    domain = 0:1.5,
    samples = 200,
    xlabel = {$x$},
    ylabel = {$f$},
    every axis x label/.style={at={(ticklabel* cs:1.01)},anchor=west,},
    every axis y label/.style={at={(ticklabel* cs:1.01)},anchor=south,},      ]
        \draw[red, thick] (axis cs: 0, 0.2) circle [radius=4];
   \draw[red, ultra thick, fill] (axis cs: 0, 0) circle [radius=3];
   \addplot [ultra thick, red ] coordinates {(0.02,0.22) (1.48,1.68)};
    \draw[red, ultra thick, fill] (axis cs: 3,3.6) circle [radius=4];
    \addplot [ultra thick, red ] coordinates { (1.52,1.92) (2.98,3.38)};
   \draw[red, thick] (axis cs: 3, 3.4) circle [radius=4];
   \draw[red, thick] (axis cs: 1.5, 1.9) circle [radius=4];
   \draw[red, thick] (axis cs: 1.5, 1.7) circle [radius=4];
\end{axis}
\end{tikzpicture}
\caption{}
\label{fig7d}%
\end{subfigure}%

\caption{(a) The graph of the minimizing configuration (inverse deformation) starting with $h_{(1)}$. (b) The graph of the minimizing configuration starting with $h_{(2)}$. (c) The deformation whose  inverse is depicted in (a).  (d) The deformation whose inverse is depicted in (b). Note the discontinuities corresponding to cracks.}
\label{fig7}
\end{figure}
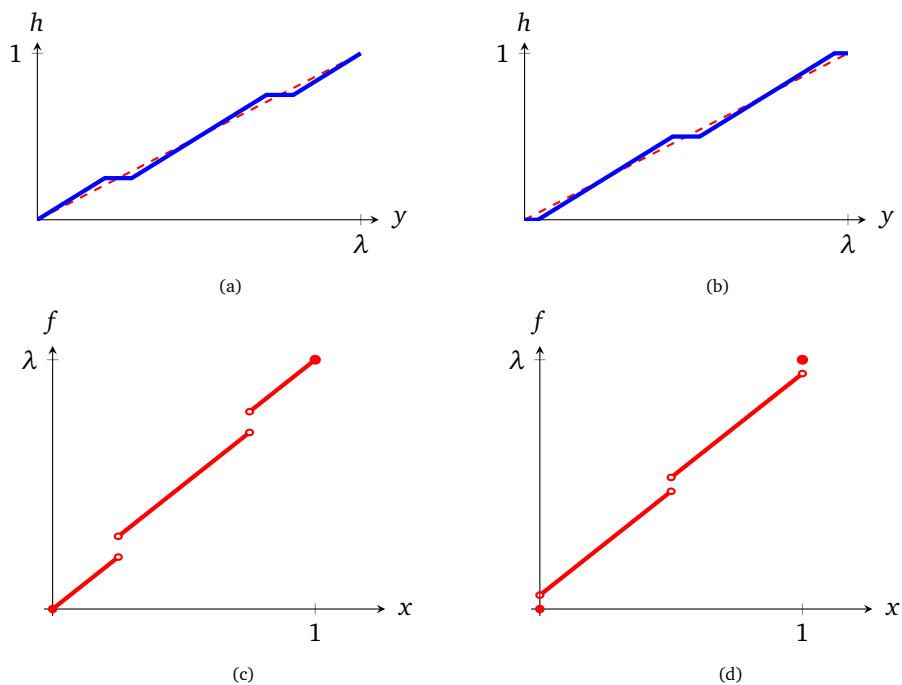%

The number of cracks clearly increases as $\lambda$ increases. For example, in Figure \ref{fig8} we give the number of cracks corresponding to the 
values of $\lambda\in (1,2)$, where we have used the same specific model as above.
\begin{figure}[h]
    \centering
    \includegraphics[scale=0.9]{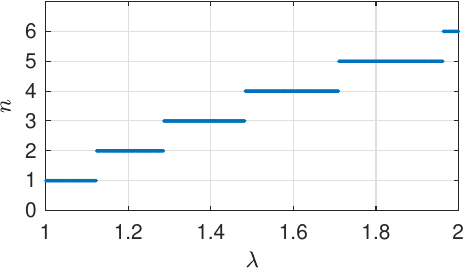}
    \caption{Number of cracks as a function of $\lambda$, when $W^{\ast }(H)=H(1-H)^{2}$
and $\mu=200$.}
\label{fig8}
\end{figure}

\section{Concluding remarks}

The properties of the minimizing configurations for the sharp-interface energies $I$ (Proposition \ref{prop2}) and $V$ (Proposition \ref{prop3}) are
in consonance with the results of \cite{RHA} and \cite{GH}, based on small  $\varepsilon>0$.  As in the former work, we find here that failure 
entails a single crack at one of the two ends, with the total energy minimized at those configurations.  A direct comparison of our results with 
those of \cite{GH} is not so straightforward.  The detailed results from the latter are numerical, with the cracked configurations only shown to be 
local energy minimizers.  On the other hand, both studies yield configurations containing multiple cracks that appear in a regularly spaced pattern, and whose number increases with load, cf. Figure \ref{fig8}. More importantly, this qualitatively agrees with observations from experiments recently reported in \cite{M2}. Here we are able to make a quantitative  prediction of the number of cracks by the explicit formula (\ref{nlambda}) for the first time, while we are assured that the configurations involved provide global energy minima. 

The formation of multiple cracks in a model combining an elastic foundation with a damage field is considered in \cite{ST}. Periodic arrangements of patterns involving multiple localized large-strain zones, interpreted as cracks, are found. The results are obtained via numerical bifurcation methods, with local stability deduced via the second variation of the discretized energy. As such, they are comparable to results presented in \cite{GH}, although truly fractured solutions are obtained in the latter. Moreover, the computed bifurcating solution branches found in \cite{ST} eventually ``return'' to the homogeneous solution path. This behavior, indicative of crack healing, is not encountered in \cite{GH} or the present work. See \cite{GH} for a discussion and references.  

Our results here imply fracture for any load  $\lambda>1$. This behavior is suggested by the results in  \cite{RHA}:  It is shown that the bar breaks along a global branch of unstable solutions at a finite load $\lambda=1+\delta$, $\delta>0$.   Referring to Figures 2 and 8 of that work, one infers that $\delta$  is an increasing function of $\varepsilon$  with $\delta=O(\varepsilon)$  as $\varepsilon \searrow 0$.  The dependence on small  $\varepsilon$ in \cite{GH} is not explored.  

A thin layer of surface effects can be observed on crack faces in both \cite{RHA} and \cite{GH}.  Indeed, it is shown that the surface energy is  $O(\varepsilon)$  as $\varepsilon \searrow 0$  in \cite{RHA}. The coefficient of $\varepsilon$ in that asymptotic result coincides with  $C_{W^{\ast }}$ defined in Proposition \ref{prop2}. Here we find that surface energy is concentrated at the crack faces, which is precisely in keeping with the Griffith picture. At the same time, the inverse-deformation approach only makes sense for finite deformations.  This contrasts with classical fracture mechanics, based on linear (infinitesimal) elasticity.

\section*{Acknowledgements}
We dedicate this work to Nick Triantafyllidis, scholar and dear friend, on the occasion of his 
$70^{\rm{th}}$ birthday.  The work of TJH was supported in part by the National Science Foundation through grant DMS-2006586, which is gratefully acknowledged. RP acknowledges the Italian National Group of Mathematical Physics of INdAM. We also thank Gokul Nair and Arnav Gupta for useful discussions.



\bibliographystyle{model1-num-names}
\bibliography{references.bib}

\begin{thebibliography}{14}
\expandafter\ifx\csname natexlab\endcsname\relax\def\natexlab#1{#1}\fi
\providecommand{\bibinfo}[2]{#2}
\ifx\xfnm\relax \def\xfnm[#1]{\unskip,\space#1}\fi
\bibitem[{Rosakis et~al.(2021)Rosakis, Healey, and Alyanak}]{RHA}
\bibinfo{author}{P.~Rosakis}, \bibinfo{author}{T.~J. Healey}, \bibinfo{author}{U.~Alyanak},
\newblock \bibinfo{title}{The inverse-deformation approach to fracture},
\newblock \bibinfo{journal}{Journal of the Mechanics and Physics of Solids} \bibinfo{volume}{150} (\bibinfo{year}{2021}) \bibinfo{pages}{104352}.
\bibitem[{Schield(1967)}]{S}
\bibinfo{author}{R.~T. Schield},
\newblock \bibinfo{title}{{Inverse deformation results in finite elasticity}},
\newblock \bibinfo{journal}{Zeitschrift für angewandte Mathematik und Physik ZAMP} \bibinfo{volume}{18} (\bibinfo{year}{1967}) \bibinfo{pages}{490--500}.
\bibitem[{Carlson and Shield(1969)}]{CS}
\bibinfo{author}{D.~Carlson}, \bibinfo{author}{T.~Shield},
\newblock \bibinfo{title}{Inverse deformation results for elastic materials},
\newblock \bibinfo{journal}{Journal of Applied Mathematics and Physics (ZAMP)} \bibinfo{volume}{20} (\bibinfo{year}{1969}) \bibinfo{pages}{261--263}.
\bibitem[{Truskinovsky(1960)}]{T}
\bibinfo{author}{L.~Truskinovsky},
\newblock \bibinfo{title}{Fracture as a phase transition},
\newblock in: \bibinfo{booktitle}{Batra, R.C. and Beatty, M.F. (eds) Contemporary research in the mechanics and mathematics of materials}, \bibinfo{publisher}{CIMNE}, \bibinfo{year}{1960}, pp. \bibinfo{pages}{322--332}.
\bibitem[{Gupta and Healey(2023)}]{GH}
\bibinfo{author}{A.~Gupta}, \bibinfo{author}{T.~J. Healey},
\newblock \bibinfo{title}{{Nucleation and Development of Multiple Cracks in Thin Composite Fibers via the Inverse-Deformation Approac}},
\newblock \bibinfo{journal}{Journal of Elasticity}  (\bibinfo{year}{2023}).
\bibitem[{Vainchtein et~al.(1999)Vainchtein, Healey, and Rosakis}]{VHR}
\bibinfo{author}{A.~Vainchtein}, \bibinfo{author}{T.~J. Healey}, \bibinfo{author}{P.~Rosakis},
\newblock \bibinfo{title}{Bifurcation and metastability in a new one-dimensional model for martensitic phase transitions},
\newblock \bibinfo{journal}{Computer methods in applied mechanics and engineering} \bibinfo{volume}{170} (\bibinfo{year}{1999}) \bibinfo{pages}{407--421}.
\bibitem[{Griffith(1921)}]{G}
\bibinfo{author}{A.~A. Griffith},
\newblock \bibinfo{title}{The phenomena of rupture and flow in solids},
\newblock \bibinfo{journal}{Philosophical Transactions of the Royal Society of London. Series A, Containing Papers of a Mathematical or Physical Character} \bibinfo{volume}{221} (\bibinfo{year}{1921}) \bibinfo{pages}{163--198}.
\bibitem[{Morankar et~al.(2023)Morankar, Mistry, Bhate, Penick, and Chawla}]{M2}
\bibinfo{author}{S.~K. Morankar}, \bibinfo{author}{Y.~Mistry}, \bibinfo{author}{D.~Bhate}, \bibinfo{author}{C.~A. Penick}, \bibinfo{author}{N.~Chawla},
\newblock \bibinfo{title}{In situ investigations of failure mechanisms of silica fibers from the venus flower basket (euplectella aspergillum)},
\newblock \bibinfo{journal}{Acta Biomaterialia} \bibinfo{volume}{162} (\bibinfo{year}{2023}) \bibinfo{pages}{304--311}.
\bibitem[{Ericksen(1975)}]{E}
\bibinfo{author}{J.~Ericksen},
\newblock \bibinfo{title}{Equilibrium of bars},
\newblock \bibinfo{journal}{Journal of Elasticity} \bibinfo{volume}{5} (\bibinfo{year}{1975}) \bibinfo{pages}{191--202}.
\bibitem[{Attouch et~al.(2014)Attouch, Buttazzo, and Michaille}]{ABM}
\bibinfo{author}{H.~Attouch}, \bibinfo{author}{G.~Buttazzo}, \bibinfo{author}{G.~Michaille}, \bibinfo{title}{Variational Analysis in Sobolev and BV Spaces}, \bibinfo{publisher}{Society for Industrial and Applied Mathematics}, \bibinfo{address}{Philadelphia, PA}, \bibinfo{year}{2014}.
\bibitem[{Braides(2002)}]{B}
\bibinfo{author}{A.~Braides}, \bibinfo{title}{{Gamma-Convergence for Beginners}}, \bibinfo{publisher}{Oxford University Press}, \bibinfo{address}{Oxford}, \bibinfo{year}{2002}.
\bibitem[{Modica(1987)}]{M}
\bibinfo{author}{L.~Modica},
\newblock \bibinfo{title}{{The gradient theory of phase transitions and the minimal interface criterion}},
\newblock \bibinfo{journal}{Archive for Rational Mechanics and Analysis} \bibinfo{volume}{98} (\bibinfo{year}{1987}) \bibinfo{pages}{123--142}.
\bibitem[{Conti et~al.(2002)Conti, Leoni, and Fonseca}]{CFL}
\bibinfo{author}{S.~Conti}, \bibinfo{author}{G.~Leoni}, \bibinfo{author}{I.~Fonseca},
\newblock \bibinfo{title}{{A $\Gamma$-convergence result for the two-gradient theory of phase transitions}},
\newblock \bibinfo{journal}{Communications on Pure and Applied Mathematics} \bibinfo{volume}{55} (\bibinfo{year}{2002}) \bibinfo{pages}{857--936}.
\bibitem[{Salman and Truskinovsky(2021)}]{ST}
\bibinfo{author}{O.~U. Salman}, \bibinfo{author}{L.~Truskinovsky},
\newblock \bibinfo{title}{De-localizing brittle fracture},
\newblock \bibinfo{journal}{Journal of the Mechanics and Physics of Solids} \bibinfo{volume}{154} (\bibinfo{year}{2021}) \bibinfo{pages}{104517}.

\end{thebibliography}


\end{document}